\documentclass[11pt]{amsart}
\usepackage{amsmath, amssymb, amsthm,fullpage, epic, eepic, graphics}
\usepackage[all]{xypic}
\newtheorem{thm}[equation]{Theorem}
\newtheorem{prop}[equation]{Proposition}

\newtheorem{cor}[equation]{Corollary}
\newtheorem{lemma}[equation]{Lemma}

\theoremstyle{definition}

\theoremstyle{remark}

\newtheorem{rem}[equation]{Remark}

\makeatletter
\renewcommand{\subsection}{\@startsection{subsection}{2}{0pt}{-3ex
plus -1ex minus -0.2ex}{-2mm plus -0pt minus
-2pt}{\normalfont\bfseries}}
\renewcommand{\subsubsection}{\@startsection{subsubsection}{2}{0pt}{-3ex
plus -1ex minus -0.2ex}{-2mm plus -0pt minus
-2pt}{\normalfont\bfseries}} \makeatother

\numberwithin{equation}{section}

\newcommand{\erem}{\hfill$\lozenge$\end{rem}\vskip 3pt }

\newcommand{\iso}{{\;\stackrel{_\sim}{\to}\;}}

\newcommand{\beq}{\begin{equation}\label}
\newcommand{\eeq}{\end{equation}}

\newcommand{\into}{\hookrightarrow}

\newcommand{\onto}{\twoheadrightarrow}

\newcommand{\Aut}{\operatorname{Aut}}

\newcommand{\en}{\enspace }

\def\R{\mathbb{R}}

\def\id{\mathrm{id}}

\def\Z{{\mathbb Z}}

\def\1{\mathbf{1}}

\begin{document}
\title{\qquad\en Looping of the numbers game and the alcoved hypercube}
\author{Q\"endrim R. Gashi, Travis Schedler, and David Speyer}

\begin{abstract}
We study in detail the so-called \emph{looping} case of Mozes's game of numbers, which concerns the (finite) orbits in the reflection representation of affine Weyl groups situated on the boundary of the Tits cone.  We give a simple proof that all configurations in the orbit are obtainable from each other by playing the numbers game, and give a strategy for going from one configuration to another. The strategy gives rise to a partition of the finite Weyl group into finitely many graded posets, one for each extending vertex of the associated extended Dynkin diagram.  These are selfdual and mutually isomorphic, and dual to the triangulation of the unit hypercube by reflecting hyperplanes, studied by many authors. Unlike the weak and Bruhat orders, the top degree is cubic in the number of vertices of the graph.  We explicitly compute the Hilbert polynomial of the poset.
\end{abstract}

\maketitle

\section{Introduction}

\subsection{The numbers game}

Mozes's game of numbers \cite{mozes}, which originated from (and generalizes) a 1986 IMO problem, has been widely studied (cf. \cite{Pro-bru, Pro-min, DE, Erik-no1, Erik-no2, erikconf,   Erik-no3, Erik-no4, eriksson, Wild-no1, Wild-no2}), and yields useful algorithms for computing with the root systems and reflection representations of Coxeter groups (see \cite[\S 4.3]{BB} for a brief summary).

We briefly recall the numbers game. Consider a Coxeter group associated to generators $s_i, i \in I$, and relations $(s_i s_j)^{n_{ij}}$, for $n_{ij} =n_{ji} \in \Z_{\geq 1} \sqcup \{\infty\}$ ($n_{ii} = 1$ for all $i$, and $n_{ij} \geq 2$ for $i \neq j$). We associate to this an unoriented graph $\Gamma$ with no loops and no multiple edges, with vertex set $I$, such that two vertices $i,j$ are adjacent if $n_{ij} \geq 3$. Consider also a choice of Cartan matrix $C = (c_{ij})$ such that $c_{ii} = 2$ for all $i$, $c_{ij} = 0$ whenever $i$ and $j$ are not adjacent, and in the case $i,j$ are adjacent, $c_{ij}, c_{ji} < 0$ and either $c_{ij} c_{ji} = 4 \cos^2(\frac{\pi}{n_{ij}})$ (when $n_{ij}$ is finite), or $c_{ij} c_{ji} \geq 4$ (when $n_{ij} = \infty$).

The \emph{configurations} of the game consist of vectors $\R^I$, considered as labelings of the graph $\Gamma$ by numbers called \emph{amplitudes}.  The moves of the game are as follows: for any vector $v = (v_i)_{i \in I} \in \R^I$ and any vertex $i \in I$ such that $v_i < 0$, one may perform the following move, called \emph{firing the vertex $i$}: $v$ is replaced by the new configuration $f_i(v)$, defined by
\begin{equation}
f_i(v)_j =
\begin{cases}
-v_i, & \text{if $j = i$}, \\
v_j - c_{ij} v_i, & \text{if $j$ is adjacent to $i$}, \\
v_j, & \text{otherwise}.
\end{cases}
\end{equation}
The entries $v_i$ of the vector $v$ are called \emph{amplitudes}.  The game terminates if all the amplitudes are nonnegative.  Let us emphasize that \emph{only negative-amplitude vertices may be fired}.

Note that the operation $f_i$ is nothing but the action of the simple reflection $s_i \in W$ if $\R^I$ is in the basis of the fundamental coweights (or weights).

The following summarizes some basic known results (where $\cdot$ denotes the usual dot product in $\mathbb{R}^I$, and in the case of affine Coxeter groups, $\delta \in \mathbb{Z}^I$ denotes a generator of the kernel of the Cartan form, in the basis of simple roots; the precise definition is recalled in \S 2):

\begin{thm}
\begin{enumerate}
\item[(i)] \cite{mozes,eriksson} If the numbers game terminates, then it must terminate in the same number of moves and at the same configuration regardless of how it is played.
\item[(ii)] In the finite Coxeter group case, the numbers game must terminate.
\item[(iii)] \cite{erikconf} In the affine case, the numbers game terminates if and only if $v \cdot \delta > 0$.
\item[(iv)] \cite{erikconf} Whenever the numbers game does not terminate, it reaches infinitely many distinct configurations, except for the affine Coxeter group case where $v \cdot \delta = 0$, in which case only finitely many configurations are reached \emph{(}i.e., the game ``loops''\emph{)}.
\end{enumerate}
\end{thm}
In particular, we will be concerned here with the \emph{looping} case.

We explain briefly how our results relate to standard constructions in
the theory of Coxeter groups. The configuration space $\R^I$ is the
reflection representation of the Coxeter group $W$.  The subset of
$\R^I$ where the numbers game terminates is called the Tits cone. The
Tits cone is naturally divided into simplicial cones, with the maximal
cones labeled by the elements of $W$ (see \cite[Sections 4.3 and
4.9]{BB}).  Beginning with a point in the interior of one of these
maximal cones, the cones we travel through form a descending chain in
the \emph{weak order} (whose definition we recall in \S \ref{the
  poset}); the restriction of firing only negative amplitudes means
that we only move downward. So results (i) and (ii), in part, say that
weak order is graded and has a unique minimal element.

Our results study the affine case, not within the Tits cone, but at
its boundary. In the affine case, the boundary of the Tits cone is a
hyperplane, divided into finitely many simplicial cones. The maximal
cones are indexed by the elements of a finite Coxeter group,
$W_0$. However, our problems do not reduce to the numbers game on
$W_0$, but instead reveal several new and interesting combinatorial
structures.

\subsection{Motivation and results}

The original motivation of this paper was the following question: in the affine case with $v \cdot \delta = 0$, can one always return to the initial configuration?  This was asserted to be true in \cite{erikconf},\footnote{Also, there it was asserted that there is a way to play the numbers game that passes through all configurations in the Weyl orbit of the vector $v$ exactly once (i.e., that a Hamiltonian cycle exists in the directed graph whose vertices are this Weyl orbit and directed edges are moves of the numbers game). We do not have either a proof or counterexample to this assertion.} but a proof was not provided. Our first goal is to provide a simple proof in the affirmative.  In fact, we prove more: we give in \S \ref{strloopsec} a \emph{strategy} for going from any configuration $v$ to any element of its Weyl orbit in a number of moves cubic in the number of vertices, and our strategy is optimal in certain cases.

In the process, we find (in \S \ref{posetsec}) that, following our strategy, the graph of obtainable configurations leads to a canonical decomposition of the finite Weyl group associated to our graph into graded selfdual posets, one for each extending vertex, whose maximal degree is cubic in the number of vertices (unlike the weak order poset, which has quadratic degree in the number of vertices). The vertices are a canonical choice of coset representatives modulo the subgroup which acts by automorphisms on the extended Dynkin graph. Moreover, we show that the graph of this poset coincides with the dual of the triangulation of the unit hypercube in the reflection representation of the affine Weyl group. This triangulation has been studied in many places, notably recently in \cite{LamPost} in type $A$.

Finally, in \S \ref{hilbpolysec}, we compute the Hilbert polynomial of this poset, thus significantly strengthening our initial result, and give explicit formulas.  Going from the lowest to top degree element of the poset involves a canonical involution of the extended Dynkin graph which we also compute. In \S \ref{combintsec}, we give a combinatorial interpretation and proof of this formula in the type $A$ cases.

Evaluating the Hilbert polynomial at $1$ in two ways yields a curious identity (which was unknown to us): Let $W_0$ be a finite Weyl group associated to a Dynkin graph $\Gamma_0$ with vertex set $I_0$, and let $W, \Gamma, I$ be the corresponding affine Weyl group and extended Dynkin graph, with $I = I_0 \sqcup \{i_0\}$. Let $m_i$ be the Coxeter exponents of $W_0$. Then,
\begin{equation}\label{curid}
\prod_i (m_i(m_i+1)) = \#(\text{extending vertices of $\Gamma$}) \cdot \prod_{i \in I} l(t_i),
\end{equation}
where the elements $t_i \in W$ are those that take the dominant
chamber of $\mathcal{H} := \{v \in \R^{\widetilde I} \mid \delta \cdot
v = 1\}$ (i.e., the locus where all coordinates are nonnegative) to
its translate by $\omega^{(i)} - \delta_i \omega^{(i_0)}$, for
$\omega^{(i)}$ is the $i$-th fundamental coweight (which is the $i$-th
basis vector of our configuration space $\R^I$), and $l(t_i)$ is the
length of $t_i$, i.e., the minimum number of simple reflections whose
product is $t_i$.

\subsection{Acknowledgements}
We thank T. Lam for essential discussions about the Hilbert
polynomials of our posets. The first author is an EPDI fellow, the
second author is an AIM fellow, and the third author is a Clay
research fellow.  The first two authors were supported by Clay Liftoff
fellowships.  The second author was also partially supported by the
University of Chicago's VIGRE grant. We thank the University of
Chicago, MIT, and the Isaac Newton Institute for Mathematical Sciences
for hospitality.

\section{Preliminaries on affine Coxeter groups}
In this note, we will be concerned with the numbers game when $\Gamma$
is an extended Dynkin graph, and $C$ is the standard integral matrix
associated to it (we can generalize this to the case where $C$ is
nonintegral but satisfies $c_{ij} = c_{ji}$ whenever $n_{ij}$ is odd:
see Remark \ref{nonintrem}). In the simply-laced case (types
$\widetilde{A_n}, \widetilde{D_n}$, or $\widetilde{E_n}$), this just
means that all the entries of $C$ are $2$ or $-1$.  In all other
types, $\Gamma = \Gamma'/S$ where $\Gamma'$ is a simply-laced graph,
$S$ is a (nontrivial) subgroup of $\Aut(\Gamma')$, and the matrix $C =
C_\Gamma = (c_{ij})$ is obtained from the matrix $C_{\Gamma'}$ as
above by the usual folding procedure, i.e., for $i$ and $j$ adjacent,
$-c_{ij}$ is equal to the number of elements of $S$ which fix $i$
divided by the number which fix both $i$ and $j$.

The interesting phenomena already occur in the simply-laced case, and the reader can assume this to be the case if preferred.

We will make use of the root systems associated to Dynkin and extended Dynkin graphs. Let $\Delta, \Delta_+$ be the set of roots and positive roots, respectively. We will view $\Delta, \Delta_+ \subset \Z^I$ in the basis of simple roots.  Then, the dot product between roots and configuration vectors (viewed as coweights) is the canonical pairing.

To be precise, by roots we mean what are sometimes called real roots, i.e., the images of the simple roots under the Coxeter group action dual to the $f_i$ action: $s_i \alpha = \alpha - \langle \alpha, \alpha^{(i)} \rangle \alpha^{(i)}$, where $\alpha^{(i)} \in \Delta_+ \subset \Z^I$ is the $i$-th simple root and $\langle \,, \rangle$ is the Cartan form, $\langle \alpha^{(i)}, \alpha^{(j)} \rangle = c_{ij}$.

For an extended Dynkin graph $\Gamma$, a vertex is called \emph{extending} if the complement of the graph is a Dynkin graph with respect to which the given vertex is the extending vertex. Finally, we will make use of the element $\delta \in \Z_+^I$, uniquely given so that $\delta_i = 1$ at all extending vertices of the graph, and $\langle \delta, \alpha \rangle = 0$ for all $\alpha \in \Delta$.

We emphasize that $\Delta$ is the set of roots for the affine Coxeter group.
On the occasion that we need to refer to the root system of the associated finite Coxeter group, we write $\Delta^0$; then $\Delta^0_+$ will denote the set of positive roots from $\Delta^0$.

\section{Strong looping of the numbers game}\label{strloopsec}
In this section, we prove and generalize the following:
\begin{thm} \label{strongloopthm} Whenever the numbers game loops, one can always return to the initial configuration.
\end{thm}
Recall that the numbers game loops if and only if the Coxeter group is affine and the initial configuration $v$ satisfies $\delta \cdot v = 0$. The theorem can be re-expressed as: if $\delta \cdot v = 0$, then for every element $g$ of the Coxeter group of the graph, one can go from $v$ to $g v$ by playing the numbers game.

As in the introduction, let $C$ be integral and given by an
identification $\Gamma \cong \Gamma'/S$ where $\Gamma'$ is
simply-laced. Playing the numbers game on $\Gamma$ is equivalent to
playing on the $S$-invariant configurations on $\Gamma'$.  Making this
choice of $C$ does not affect the validity of the theorem.

We will prove a stronger result, which gives a \emph{strategy} for
obtaining $v$ from $u$ whenever $v$ and $u$ are in the same (affine) Weyl
orbit, and a bound on the number of moves required (which will be
\emph{cubic} in the number of vertices of $\Gamma$).

To explain this, first note that, whenever $\delta \cdot v = 0$, 
$v$ is uniquely determined by its restriction to any subgraph
$\Gamma_0 \subsetneq \Gamma$ obtained by removing exactly one
vertex. We now fix a subgraph $\Gamma_0$ such that $\Gamma$ is the
extending graph of $\Gamma_0$ (i.e., $\Gamma_0$ corresponds to a
finite Coxeter group, and $\Gamma$ corresponds to the associated
affine Coxeter group).  Let $W_0$ be the finite Coxeter group
associated to $\Gamma_0$. For any configuration $v$ on $\Gamma$, let
$v_0$ denote the restriction to $\Gamma_0$. Similarly, for any
configuration $u$ on $\Gamma_0$, let $\widetilde {u}$ be the unique
configuration on $\Gamma$ such that $\delta \cdot \widetilde {u} = 0$
and $(\widetilde {u})_{0} = {u}$.  Then, for any configuration $v$ on
$\Gamma$ such that $\delta \cdot v = 0$, and any element $g \in W_0$,
the configuration $g v := \widetilde{g v_0}$ makes sense.  Moreover,
any sequence of moves of the numbers game can be represented as
applying an element of $W_0$, since firing the extending vertex is the
same (when $\delta \cdot v = 0$) as applying the reflection about the
maximal root of $\Gamma_0$.

Thus, we can reformulate the theorem as follows: for
any two Weyl chambers $C, C'$ of $W_0$, and every vector $v \in C$,
there exists a way to play the numbers game on $\Gamma$ to take
$\widetilde{v}$ to a vector whose restriction to $\Gamma_0$ is in
$C'$.

Next, note that one can always play the numbers game on $\Gamma_0$
until all the amplitudes on $\Gamma_0$ are nonnegative: this yields
the unique configuration obtainable from the original one which is in
the \emph{dominant} Weyl chamber. Similarly, by playing the numbers
game in reverse, one sees that every configuration is obtainable from
one in the \emph{antidominant} Weyl chamber.  Thus, to prove the
theorem, it is enough to show that one can take every vector of the
dominant Weyl chamber to a vector of the antidominant Weyl
chamber. Moreover, for this it suffices to take any single vector in
the interior of the dominant Weyl chamber: we will take the vector
$\widetilde \rho$, where $\rho$ is the vector on $\Gamma_0$ whose
amplitude is $1$ at all vertices.\footnote{$\rho$ corresponds to
  one-half the sum of all positive coroots of $\Gamma_0$ (or the sum
  of all fundamental coweights), and is important in Lie theory.}

Let $I$ be the vertex set of $\Gamma$ and $I_0$ be the vertex set of the
fixed subgraph $\Gamma_0$. Let $i_0 \in I \setminus I_0$ be the
extending vertex of $\Gamma_0$.  Set $\rho^{(i_0)} :=
\widetilde{\rho}$. For any other extending vertex $i \in I$ (i.e.,
such that the extended Dynkin graph of $\Gamma \setminus\{i\}$ is
$\Gamma$), let $\rho^{(i)} \in \Z^I$ also denote the configuration
with $\delta \cdot \rho^{(i)} = 0$ such that $(\rho^{(i)})_j = 1$ for
all $j \neq i$ (i.e., $\rho^{(i)}$ is the image of $\rho^{(i_0)}$
under an automorphism of $\Gamma$). We may now strengthen Theorem
\ref{strongloopthm}:

\begin{thm} \label{stratthm} For any $g \in W_0$, beginning with
  $\widetilde{g \rho}$ and playing the numbers game by arbitrarily
  firing vertices of amplitude $< -1$, one obtains $-\rho^{(i)}$ for
  some extending vertex $i \in \Gamma$. Regardless which moves are
  chosen, the total number of moves is the same, as is the vertex $i$.
  Moreover, this is the fastest way, under the numbers game, to get
  from $\widetilde{g \rho}$ to a configuration of the form
  $-\rho^{(i')}$. Finally, taking $g = 1$, the resulting map $i_0
  \mapsto i$ is an involution on the extending vertices of $\Gamma$.
\end{thm}

Moreover, the \emph{score vector} is the same regardless of the moves chosen: this means the configuration vector $s$ on $\Gamma$ which records, at each vertex $i$ of $\Gamma$, the sum of negative all the amplitudes fired in the course of playing the game.

To prove the theorem, we will use two lemmas:

\begin{lemma} \label{grholem}
Let $\sigma \in \Z^I$ have the following properties:
\begin{enumerate}
 \item[(i)] $\sigma$ is in the boundary of the Tits cone, \emph{i.e.}, $\sigma \cdot \delta =0$.
\item[(ii)] For each $i \in I$,  $\sigma_i \geq -1$.
\item[(iii)] For every root $\alpha \in \Delta$,  $\sigma \cdot \alpha \neq 0$.
\end{enumerate}
Then $\sigma = -\rho^{(i)}$ for some extending vertex $i$.
\end{lemma}

Note in particular that, for any $g \in W_0$, the vector $\widetilde{g \rho}$ satisfies~(i) and~(iii).

\begin{proof}
  Recall that $\Delta^0$ denotes the root system of $W_0$. The simple
  roots of $\Delta$ are the simple roots $\alpha^{(i)} \in
  \Delta_{I_0}$ for $i \in I_0$, together with $\alpha^{(i_0)} =
  \delta - \beta_{\mathrm{long}}$, where $\beta_{\mathrm{long}}$ is
  the longest root in $\Delta^0$.  Using (i) and (ii), we deduce that
  $\sigma \cdot \gamma \geq -1$ for $\gamma$ in $R := \{ \alpha^{(i)}
  \}_{i \in I_0} \cup \{ - \beta_{\mathrm{long}} \}$.

Taking inner product with $\sigma$ yields a linear function on $\Delta^0$, which, by condition~(iii), is not zero on any root.  Let $\mathcal{N}$ be the subset of $\Delta^0$ whose inner product with $\sigma$ is negative; so  $\mathcal{N} = - w \Delta^0_+$ for some $w \in W_0$.  For any $\gamma \in \mathcal{N}$,  $\gamma = - \sum_{i \in I_0} b_i (w \alpha^{(i)})$ for some nonnegative integers $b_i$, and then $\sigma \cdot \gamma \leq - \sum b_i$.  If $\gamma$ is an element of $\mathcal{N}$ which is not of the form $- w \alpha^{(i)}$ then we deduce that $\sigma \cdot \gamma < -1$, and thus $\gamma \notin R$.

Combining the observations of the last two paragraphs, we see that $R \subseteq w \Delta^0_{+} \cup \{ - w \alpha^{(i)} \}_{i \in I_0}$. Suppose, for the sake of contradiction, that $-w \alpha^{(j)} \not \in R$ for some $j \in I_0$.  Then $w^{-1} R$ lies in the closed half-space whose boundary is spanned by the simple roots other than $\alpha^{(j)}$.  But there is a positive linear dependence between the elements of $R$, which is a contradiction.

We deduce that $R \supset  \{ - w\alpha^{(i)} \}_{i \in I_0}$. So $\sigma \cdot \beta$ is $-1$ for all but one element $\beta \in R$. Viewing $\sigma$ as an element of $\R^I$, this says that all but one coordinate is $-1$. Let that one coordinate be $i$. Consider the simple roots of $W$ associated to $I \setminus \{ i \}$; modulo $\delta$, these roots form a simple root system for $W_0$ (namely, $ \{-w \alpha^{(i)} \}_{i \in I_0}$). So, $i$ is an extending vertex and $\sigma = -\rho^{(i)}$.
\end{proof}

\begin{lemma} \label{termination} Beginning with any configuration $v$, suppose that is possible to fire $r$ vertices \emph{(}counted with multiplicity\emph{)} of amplitude $< -1$ until there are none left. Consider any other sequence of firing $s$ vertices of amplitude $< -1$.  Then $s \leq r$ and this sequence can be extended to a sequence of $r$ firings of vertices of amplitude $< -1$ which terminates at the same configuration.
\end{lemma}

\begin{cor} Beginning with any configuration $v$, suppose there are two ways of firing vertices of amplitude $< -1$ until there are none left.  Then these two paths are of the same length, and terminate at the same configuration.
\end{cor}

\begin{proof}[Proof of Lemma~\ref{termination}] This is very similar to the proof of strong convergence in the numbers game, so we will be brief.  Our proof is by induction on $r$.  Notice that, if $v_i < -1$ and $v_j < -1$ then the $j$ coordinate will still be $< -1$ after the $i$-vertex is fired.  Thus, if $r=1$ then only one coordinate of $v$ is $< -1$ and the claim is obvious.

  For larger $r$, let $v \to_i \sigma \to \cdots \to \zeta$ be our
  sequence of length $r$, beginning by firing $i$, and let $v \to_j
  \tau \to \cdots \to \zeta'$ be the sequence of length $s$, beginning
  by firing $j$.  Then $v_i$ and $v_j$ are both $< -1$.  Consider
  alternately firing $i$ and $j$ until both coordinates are $\geq
  -1$. (If this never occurs, then no path from $v$ can terminate.)
  Let $\pi$ be the configuration when both coordinates become $\geq
  -1$. We claim that this configuration does not depend on whether we
  fire $i$ or $j$ first and the two paths $v \to \sigma \to \cdots \to
  \pi$ and $v \to \tau \to \cdots \to \pi$ have the same length $t =
  n_{ij}$.  In fact, $(\pi_i, \pi_j) = (-v_i, -v_j)$ if $n_{ij}$ is
  even, and $(\pi_i, \pi_j) = (-v_j, -v_i)$ if $n_{ij}$ is odd.  It is
  clear that applying $n_{ij}$ firings, alternating between firing
  vertex $i$ and $j$, produces the desired result, regardless of which
  reflection is applied first: we only have to prove that, along the
  way, only vertices of amplitude $< -1$ are fired.  This follows
  because, by usual strong convergence of the numbers game, only
  negative-amplitude vertices are fired; these amplitudes must be of
  the form $\alpha \cdot (v_i, v_j)$ where $\alpha$ is a positive root
  for the restriction of the diagram to vertices $i$ and
  $j$. Therefore, $\alpha$ is a vector with nonpositive integral
  entries, implying that the dot product is indeed $< -1$.

By induction, we can extend $\sigma \to \cdots \to \pi$ to a path $\sigma \to \cdots \to \pi \to \cdots \to \zeta$ of length $r-1$. Tacking the second part of this path onto the path $\tau \to \cdots \to \pi$, we obtain a path $\tau \to \cdots \to \zeta$ of length $r-1$. Using induction again, we can complete $\tau \to \cdots \to \zeta'$ to a path $\tau \to \cdots \to \zeta' \to \cdots \to \zeta$. Tacking $v \to_j \tau$ on the beginning of this path, we are done.
\end{proof}

Note that there is nothing in the above lemma that requires the number $-1$: the same is true for any negative number, and neither the number nor $v$ need be integral.

\begin{proof}[Proof of Theorem \ref{stratthm}]

We begin by showing that there is a path from any $\rho^{(i)}$ to some $- \rho^{(j)}$ by firing only vertices of amplitude $< -1$. Start at $\rho^{(i)}$ and fire vertices of amplitude $< -1$ in any manner.  Since the $W$ orbit of $\rho^{(i)}$ is finite, either we will reach a configuration with no vertices of amplitude $< -1$, or we will repeat a configuration. In the former case, by Lemma \ref{grholem}, we are done.  In the latter case, we have a path of the form $\rho^{(i)} \to \cdots \to \sigma \to \cdots \to \sigma$. Reversing this path and negating all the configurations, we obtain a path $- \sigma \to \cdots \to - \sigma \to \cdots \to - \rho^{(i)}$ which only fires vertices of amplitude $< -1$.  So we have two paths of different lengths from $-\sigma$ to $-\rho^{(i)}$, contradicting Lemma \ref{termination}. Note that we now have enough to prove Theorem \ref{strongloopthm}.

We continue with the proof of Theorem \ref{stratthm}. Let $\sigma$ be of the form $\widetilde{g \rho}$.  Because the numbers game for $W_0$ terminates, it is possible to get from $\sigma$ to $\rho^{(i_0)}$ by firing only vertices of negative amplitude.  By the result we just established, it is possible to get from $\sigma$ to some $- \rho^{(i)}$ by firing only vertices of negative amplitude.  Let the length of a shortest path from $\sigma$ to some $-\rho^{(i)}$ be $\ell$; we must show that this path involves only firing vertices of amplitude $< -1$. Our proof is by induction on $\ell$; the base case $\ell=0$ is obvious. If $\sigma$ is of the form $-\rho^{(i')}$, we are clearly done. If not then, by Lemma \ref{grholem},  $\sigma_j < -1$ for some $j$.

Let the first step of our path go from $\sigma \to \sigma'$, firing $k$.  The rest of the path, from $\sigma'$ to $-\rho^{(i)}$, must be a shortest path from $\sigma'$ to any $-\rho^{(i')}$ and hence, by induction, must only involve firing vertices of amplitude $< -1$. Let $\sigma_k = a$. We want to show that $a < -1$.

Starting at $\sigma$, alternately fire $j$ and $k$ until the $j$ and $k$ coordinates are both positive. The resulting configuration, $\tau$, is the same whether we fire $j$ or $k$ first, and the length of the resulting path from $\sigma$ to $\tau$ is $n_{jk}$ in either case. On the route from $\sigma'$ to $\tau$, all the vertices fired are of amplitude $< -1$. Using Lemma \ref{termination}, we can fire $\ell - n_{jk}$ more vertices of amplitude $< -1$ to get from $\tau$ to $-\rho^{(i)}$. Now, consider the path $\sigma \to \cdots \to \tau' \to \tau \to \cdots \to (-\rho^{(i)})$ which starts by firing $j$. Then the vertex which is fired when going from $\tau'$ to $\tau$ has amplitude $a$. The path from $\tau'$ to $(-\rho^{(i)})$ must be shortest possible (or there is a shorter path from $\sigma$ to $-\rho^{(i)}$). So, by induction, every vertex which is fired in this path has amplitude $< -1$. In particular, $a < -1$ and the original path fires only vertices of amplitude $< -1$.

It remains only to show that, for $\sigma=\rho^{(i_0)}$, the resulting map $\iota: i_0 \mapsto i$ on extending vertices of $\widetilde \Gamma$ is an involution.  For this, note that if a firing sequence of vertices of amplitude $< -1$ takes $\rho^{(i_0)}$ to $-\rho^{(i)}$, then the sequence in reverse takes $\rho^{(i)}$ to $-\rho^{(i_0)}$, also by vertices of amplitude $< -1$.
\end{proof}

\begin{rem} One may easily compute the involution $\iota$ appearing above:

\begin{prop}\label{ioprop}
The involution $\iota$  is   trivial for exactly the graphs
\begin{equation}\label{iotriveq}
\widetilde{A_{2m}}, \widetilde{B_{4m-1}}, \widetilde{B_{4m}}, \widetilde{D_{4m}}, \widetilde{D_{4m+1}}, \widetilde{E_6}, \widetilde{E_8}, \widetilde{F_4}, \widetilde{G_2}, \quad m \geq 1.
\end{equation}
For graphs $\widetilde{A_{2m-1}}$, $\iota$ is the involution sending every vertex in $\Gamma$ to its antipodal vertex.  For type $\widetilde{D_{4m+2}}, \widetilde{D_{4m+3}}$, $\iota$ interchanges exterior vertices which are adjacent to a common internal vertex, and is the identity on interior vertices.  For types $\widetilde{B_{4m+1}}, \widetilde{B_{4m+2}}, \widetilde{C_{m+1}}, \widetilde{E_7}$, $\iota$ is the unique nontrivial automorphism of the graph.
\end{prop}
We will explain the proof following Lemma \ref{trlem} in the next section.
\end{rem}

\begin{rem}\label{nonintrem}
  The results of this section generalize to the case when the Cartan
  matrix $C$ is not necessarily integral, but satisfies $c_{ij} =
  c_{ji}$ whenever $n_{ij}$ is odd.  This is because, in this case,
  every positive root $\alpha$ is obtained from a simple root by
  simple reflections that only increase coordinates in the basis of
  simple roots.  Hence, for every positive root $\beta$ and every
  configuration $v$ such that $v \in \R_{\leq -1}^I$, we must have $\beta
  \cdot v \leq -1$, with equality holding only if $\beta = \alpha^{(j)}$
  is a simple root such that $v_j = -1$.  Then, all of the statements
  of results above remain unchanged, except that, in Lemma
  \ref{grholem}, we should let $\sigma$ be in the Weyl orbit of a
  vector $\tau \in \R^I$ such that $\tau_i \geq 1$ for all $i \in
  I_0$, leaving the conditions (i)--(iii) unchanged.  The proofs are
  only changed to replace the integrality by the above fact.

Note that the results do not hold in general without the assumption that $c_{ij} = c_{ji}$ for odd $n_{ij}$: for instance, a counterexample to Lemma \ref{termination} is obtained if we consider the Dynkin graph $\Gamma_0 = A_2$ with Cartan matrix $\begin{pmatrix} 2 & -2 \\ -\frac{1}{2} & 2 \end{pmatrix}$ and the vector $v = (-2, -2)$, then one firing sequence yields $(-2, -2) \rightarrow (-3, 2) \rightarrow (3, -4) \rightarrow (1, 4)$, firing only vertices of amplitude $< -1$, whereas the other firing sequence yields $(-2, -2) \rightarrow (2, -6) \rightarrow (-1, 6)$, and then we would have to fire a vertex of amplitude $\geq -1$ to obtain the same configuration as in the previous firing sequence. As a result, the configuration $(-2, -2, 3\sqrt{2})$ on the diagram $\Gamma = \widetilde{A_2}$ with Cartan matrix $\begin{pmatrix} 2 & -2 & -1 \\ -\frac{1}{2} & 2 & -1 \\ -1 & -1 & 2\end{pmatrix}$ contradicts Theorem \ref{stratthm}.  Moreover, the vector $(-1, 6, 2 \sqrt{2})$ contradicts Lemma \ref{grholem}, as generalized in the previous paragraph.
\end{rem}

\section{The resulting poset and geometric interpretations}\label{posetsec}

In this section, in Proposition \ref{strloopcmpp}, we relate playing
the numbers game on the boundary of the Tits cone, where we fire only
vertices of amplitude $< -1$ beginning from $\rho^{(i_0)}$, with
playing the numbers game slightly inside the Tits cone where we fire
any vertex with negative amplitudes beginning with configurations of
the form $\rho^{(i_0)} + \kappa \omega^{(i_0)}$, for appropriate
values of $\kappa >0$.  Recall that $\omega^{(i)}$ is the
$i$-th fundamental coweight, which is also the $i$-th basis vector of
configuration space $\R^I$.
 
We then study the posets $P_i$ of configurations obtainable from $\rho^{(i)}$ by firing vertices of amplitude $< -1$, and see how these give rise to a canonical decomposition of the finite Weyl group $W_0$ into isomorphic graded selfdual posets, $W_0^{(i)}$, one for each extending vertex, whose maximal degree is cubic in the number of vertices (unlike the weak order poset, which has quadratic degree in the number of vertices). Then, using Proposition \ref{strloopcmpp}, we show how the poset $W_0^{(i_0)}$ may be identified with an interval under the left weak order in the affine Weyl group.

Finally, we also show that the graph of this poset coincides with the dual of the triangulation of the unit hypercube in the reflection representation of the affine Weyl group. This triangulation has been studied in many places, notably recently in \cite{LamPost} in type $A$.

\subsection{Relation to strong convergence, and number of moves required}

Here, we compare firing only amplitudes $< -1$ beginning from $\rho^{(i_0)}$ with firing \emph{any} negative amplitudes beginning from the configuration $u_\kappa := \rho^{(i_0)} + \kappa \omega^{(i_0)}$, for certain positive values of $\kappa$. Using this, strong convergence of the numbers game for $u_\kappa$ will explain why $\rho^{(i_0)}$ can reach $-\rho^{(i)}$ and why the strategy of (arbitrarily) firing vertices of amplitude $< -1$ is the fastest way to do so.  This will also allow us to easily compute the number of moves required.

Let $\beta_{\mathrm{long}} \in (\Delta_{\Gamma_0})_+$ be the longest root of $(\Delta_{\Gamma_0})_+$.

\begin{prop} \label{strloopcmpp} If $\kappa \in (1, 1 + \frac{1}{l(\beta_{\mathrm{long}})-1})$, then a sequence of vertices $i_1, i_2, \ldots$ is a valid firing sequence for the numbers game beginning with $u_\kappa$ if and only if it is a firing sequence of amplitudes $< -1$ beginning with $\rho^{(i_0)}$.  Moreover, the final configurations reached under a maximal such sequence are $(\kappa - 1 )\rho^{(\iota(i_0))} + \kappa \omega^{(\iota(i_0))}$ and $-\rho^{(\iota(i_0))}$, respectively, for some fixed involution $\iota \in \Aut(\Gamma)$.
\end{prop}

The proposition relies on the following basic, and probably well known, lemma (whose proof we supply for the reader's convenience). For each vertex $j \in I_0$, let $T_j: \R^I \rightarrow \R^I$ be the ``translation'' element of the form
\begin{equation}\label{translations}
T_j (v) = v + (\delta \cdot v) (\omega^{(j)} - \delta_j \omega^{(i_0)}).
\end{equation}
Let $P^\vee_{I_0}$ be the coweight lattice (with basis the fundamental
coweights $\omega^{(i)}, i \in I_0$) and let $Q^\vee_{I_0} := \langle
\alpha^\vee: \alpha \in \Delta_{I_0}\rangle \subset P^\vee_{I_0}$ be
the coroot sublattice. In the basis of fundamental coweights,
$P^\vee_{I_0} = \Z^{I_0} \supseteq \langle \alpha^\vee: \alpha \in
\Delta_{I_0} \rangle$, where $(\alpha^\vee)_i = \langle \alpha,
\alpha^{(i)} \rangle$ for all $i \in I_0$.

\begin{lemma} \label{trlem} For any vertex $j \in I_0$, there is a unique element $t_j \in W$ whose action on $\R^I$ is of the form $t_j = T_j \circ \gamma_j$, where $\gamma_j: \R^I \rightarrow \R^I$ is a permutation of coordinates corresponding to an automorphism of the graph $\Gamma$.  Moreover, the map $\omega^{(j)} \mapsto \gamma_j$ induces a group monomorphism $P^\vee_{I_0} / Q^\vee_{I_0} \into \Aut(\Gamma)$.\footnote{The group $P^\vee_{I_0} / Q^\vee_{I_0}$ is well known and called the \emph{fundamental group} of the root system.}
\end{lemma}

\begin{proof}
Fix $\kappa > 0$ and let us consider the hyperplane $\mathcal{H}_\kappa := \{v \in \R^I \mid \delta \cdot v = \kappa\}$, fixed under $W$. It suffices to show that the lemma holds restricted to $\mathcal{H}_\kappa$.  Note that the triangulation of $\mathcal{H}_\kappa$ by its intersection with the Weyl chambers has the translational symmetry $T_j|_{\mathcal{H}_\kappa}$. Thus, there must exist a unique element $t_j \in W$ such that $t_j$ takes the dominant Weyl chamber (i.e., the one whose amplitudes are all nonnegative) to its translate under $T_j$.  We must therefore have that $T_j^{-1} \circ t_j$ is an isometry of the dominant Weyl chamber.  Thus, $T_j^{-1} \circ t_j$ is induced by an automorphism of $\Gamma$.

To see that this induces a homomorphism $\psi: P^\vee_{I_0}
\rightarrow \Aut(\Gamma)$, we need to show that the $\gamma_j$ all
commute with each other.  When $\Gamma$ is not of type $A$ or $D$, this is
immediate since $\Aut(\Gamma)$ is abelian.  In the case of type $A$,
an easy computation shows that, when $j$ is adjacent to $i_0$, then
$\gamma_{i_0}$ is a rotation of the diagram $\Gamma$ (moving each
vertex to an adjacent one), and it is easy to see that all the
$\gamma_{j'}$ can be obtained from products of conjugates of this one,
and therefore that the images of the $\gamma_{j'}$ generate the
abelian normal subgroup $\Z/n < D_{2n} = \Aut(\Gamma)$, in this case.
In the case of type $D$, the image of $P_{I_0}$ is a $4$-element
subgroup of $D_4$ and is hence abelian (we can also compute explicitly
that the image is abelian, like in type $A$).

Finally, we need to show that the kernel of $\psi$ is
$Q^\vee_{I_0}$. Let $f_{\beta_{\mathrm{long}}} \in W_0 \subset W$ denote the
reflection corresponding to the maximal root $\beta_{\mathrm{long}} \in
(\Delta_{\Gamma_0})_+$, so $f_{\beta_{\mathrm{long}}}(v) = v - 
(\beta_{\mathrm{long}} \cdot v)\beta_{\mathrm{long}}^{\vee}$.  Let us define
\begin{equation}
f_i' := \begin{cases} f_i, & \text{if $i \neq i_0$}, \\
                     f_{\beta_{\mathrm{long}}}, & \text{if $i = i_0$}.
\end{cases}
\end{equation}
Then, $f_{i_1} f_{i_2} \cdots f_{i_m}$ is a translation if and only if
$f'_{i_1} f'_{i_2} \cdots f'_{i_m} = 1$, and in this case, $$f_{i_1}
f_{i_2} \cdots f_{i_m} (f_{i_1}' f'_{i_2} \cdots f_{i_m}')^{-1}$$ can
be written as a product of conjugates of the translation $f'_{i_0}
f_{i_0} = \psi(\beta_{\mathrm{long}}^{\vee})=\prod_{i \in I_0} T_i^{\langle \alpha^{(i)}, \beta_{\mathrm{long}}^{\vee}
  \rangle}$.  In other words, the translations of $W$ are exactly
$\psi(Q^\vee_{I_0})$.
\end{proof}

We remark that the above lemma also allows one to prove Proposition
\ref{ioprop}, using the statement of Proposition \ref{strloopcmpp}.
Indeed, for $\kappa$ as in Proposition \ref{strloopcmpp}, we see that
$v := -\rho^{(i_0)} + \kappa \omega^{(i_0)}$ is in the same affine
Weyl orbit as $u_\kappa = \rho^{(i_0)} + \kappa \omega^{(i_0)}$
(since, by playing the numbers game on the Dynkin subgraph $\Gamma_0$,
one may go from $-\rho^{(i_0)}$ to $\rho^{(i_0)}$). Thus, $v$ is in
the affine Weyl orbit as the final configuration $u :=
(\kappa-1)\rho^{(\iota(i_0))} + \kappa \omega^{(\iota(i_0))}$, which
is in the Weyl chamber containing $\iota(\prod_{i \in I_0} T_i v)$, and
hence $u$ is obtainable from $v$ by applying $|I_0|$
translates $t_i$.  Thus, the formula for $\iota$ follows 
by computing the automorphisms $\gamma_i$,
which is not difficult (and probably well known).

\begin{proof}[Proof of Proposition \ref{strloopcmpp}]
When we play the numbers game beginning with $u_\kappa$, let us keep track of not only the resulting configuration $f_{i_m} f_{i_{m-1}} \cdots f_{i_1} (u_\kappa)$, but also the vector $f_{i_m} f_{i_{m-1}} \cdots f_{i_1} (u_x)$, where $\kappa$ is replaced by an indeterminate $x$, and such that valid firing sequences beginning with $u_x$ are, by definition, sequences that become valid when $x$ is evaluated at $\kappa$.

First, we claim that, playing the numbers game in this way from $u_x$, any amplitude of the form $-1 + a x$ that appears must have $a \geq 1$.  The reason for this is that, if $a \leq 0$, then $-1 + a \kappa < 0$, and since this is negative, it can only appear by playing the numbers game if $\alpha \cdot u_x = \alpha \cdot \rho^{(i_0)} + x (\alpha \cdot \omega^{(i_0)})= -1 + ax$ for some positive root $\alpha \in \Delta_+$, but the latter is impossible. Therefore, when playing the numbers game beginning with $u_x$, we never fire an amplitude of the form $-1 + ax$, for $a \leq 0$.  As a consequence, any firing sequence for $u_x$ which becomes valid when $x$ is evaluated at $\kappa > 1$, must also be a firing sequence for $u_0 = \rho^{(i_0)}$ where only amplitudes $< -1$ are fired.

Next, we claim that, if $\kappa \in (1, 1 + \frac{1}{l(\beta_{\mathrm{long}})-1})$, then the numbers game beginning with $u_{\kappa}$ terminates at the configuration $(\kappa - 1 )\rho^{(\iota(i_0))} + \kappa \omega^{(\iota(i_0))}$, and that playing the same moves from $u_x$ yields $y:= (x - 1 )\rho^{(\iota(i_0))} + x \omega^{(\iota(i_0))}$. By Lemma \ref{trlem}, the element $y$ is in the $W$-orbit of $u_{x}$, for some automorphism $\iota \in \Aut(\Gamma)$. If we evaluate $y$ at $x = \kappa \in (1, 1 + \frac{1}{l(\beta_{\mathrm{long}})-1})$ we get a dominant weight (in fact all coefficients are positive, so we get a regular dominant weight), and thus the numbers game can only terminate at $y$, and since $\delta \cdot u_{\kappa} = \kappa > 0$, the numbers game must terminate.

As a consequence of the previous two paragraphs, for any $\kappa \in (1, 1+ \frac{1}{l(\beta_{\mathrm{long}})-1})$, and any maximal valid firing sequence beginning at $u_{\kappa}$, the same firing sequence takes $\rho^{(i_0)}$ to $-\rho^{(\iota(i_0))}$. (Note that this gives another proof of Theorem \ref{strongloopthm}.)

It remains to show that any firing sequence from $\rho^{(i_0)}$ of
amplitudes $< -1$ is also a valid firing sequence for $u_\kappa$ with
$\kappa$ as above.  Let $N$ be the number of moves of any firing
sequence as above.  By Theorem \ref{stratthm}, any maximal firing
sequence of vertices of amplitude $< -1$ takes $\rho^{(i_0)}$ to
$-\rho^{(\iota(i_0))}$, in exactly the same number of moves, $N$.

Therefore, any firing sequence of length $m$ beginning with $\rho^{(i_0)}$ of vertices of amplitude $< -1$ must take $u_{\kappa}$ to a configuration which can win the numbers game (arrive at $(\kappa - 1 )\rho^{(\iota(i_0))} + \kappa \omega^{(\iota(i_0))}$) in $\leq N-m$ moves. As a result, the sequence must be a valid firing sequence for $u_{\kappa}$, since each move must have decreased by one the set of positive roots which pair negatively with the current configuration.
\end{proof}

In the case that $\Gamma = \widetilde{A_n}$, we may immediately deduce a formula for the number of moves required to go from $\rho^{(i_0)}$ to $-\rho^{(\iota(i_0))}$.  Indeed, this equals the number of positive roots $\alpha \in \Delta_+$ such that $\alpha \cdot u_\kappa < 0$ when $\kappa \in (1, 1 + \frac{1}{n})$.  For any positive root $\alpha \in \Z_{\geq 0}^I$ such that $\alpha_{i_0} = m$, then $\alpha = (m-1) \delta + \alpha'$, where $\alpha'$ is any positive root supported on a segment containing $i_0$, of length $\leq n-m$. Let us identify $I \cong \{0,1,2,\ldots, n\}$, with $i_0 = 0$, and with two integers adjacent if they differ by one modulo $n+1$. Then, to pick a pair of $\alpha' = \alpha^{(i)} + \alpha^{(i+1)} + \cdots + \alpha^{(n)} + \alpha^{(0)} + \alpha^{(1)} + \cdots + \alpha^{(j)}$ and the integer $0 \leq m \leq n-(n-i+j+2)$ is equivalent to picking the triple $j < j+m+1 < i$ of distinct integers in $\{1,2,\ldots,n\}$. Thus, the total number of such $\alpha$ is ${n+2 \choose 3}$.

As a result, to go from any configuration $v$ with $\delta \cdot v =
0$ to $u \in W v$, it takes no more than ${n+1 \choose 2} + {n+2
  \choose 3} = \frac{n(n+1)(n+5)}{6}$ moves. Indeed, assume without
loss of generality that $u, v \in W \rho^{(i_0)}$. Then, letting $i
\in I$ be the extending vertex such that $u$ can be obtained from
$\rho^{(i)}$ by firing vertices of amplitude $< -1$, it can take at
most ${n+1 \choose 2}$ moves to go from $v$ to $\rho^{(i)}$, and and
then at most ${n+2 \choose 3}$ moves to go from there to $u$.

We will prove a more general result (and for any extended Dynkin graph) in \S \ref{hilbpolysec} below.

\subsection{The poset obtained from the strategy}\label{the poset}

For each extending vertex $i$, let $P_i$ be the poset of configurations obtainable from $\rho^{(i)}$ by firing vertices of amplitude $< -1$.  By Theorem \ref{stratthm}, 
\begin{equation} \label{wdece}
W \rho^{(i_0)} = \{ \widetilde{w \rho}: w \in W_0\} = \bigsqcup_{i \text{ an extending vertex}} P^{(i)}.
\end{equation}
Moreover, each $P^{(i)}$ is a graded poset: this means that each element $v \in P^{(i)}$ has a well-defined degree, given by the number of firings of vertices with amplitude $< -1$ needed to go from $\rho^{(i)}$ to $v$.  They are also \emph{self-dual}, which means that the poset is isomorphic to the one where the ordering is reversed.

It is clear that the $P^{(i)}$ are all isomorphic posets. Since $W_0 \subset W$ acts freely on $\rho^{(i_0)}$, we may view \eqref{wdece} as a decomposition of $W_0$ itself into isomorphic graded posets, $W_0 = \bigsqcup_i W_0^{(i)}$.  Moreover, the isomorphism $W_0^{(i_0)} \iso W_0^{(i)}$ is nothing but $w \mapsto r^{(i)} w$, where $r^{(i)} \in W_0$ is the element such that $\widetilde{r^{(i)} \rho} = \rho^{(i)}$.

We remark that \eqref{wdece} is quite canonical. In particular, it does not depend on the choice of the dominant vector $\rho^{(i_0)}$: any element $u$ whose restriction to $\Gamma_0$ is in the interior of the dominant Weyl chamber (i.e., all amplitudes are positive) gives rise to the same decomposition \eqref{wdece}, except that $P^{(i)}$ are now defined as the graded posets of configurations obtainable along a minimal-length firing sequence from $r^{(i)} u$ to $-r^{\iota(i)} u$.  When we pass to the decomposition of $W_0$ itself into isomorphic graded posets, the result is independent of $u$.

As we saw at the end of the previous section, $W_0^{(i_0)}$ is quite different from the weak or Bruhat orders on $W_0$: rather than having at most quadratic degree in the number of vertices of the Dynkin diagram, we have \emph{cubic} degree.  The elements of the poset $W_0^{(i_0)}$ are a canonical choice of representatives of the right cosets of the group generated by the $r^{(i)}$ (which is a subset of $\Aut(\Gamma)$).

As a consequence of Proposition \ref{strloopcmpp}, the poset $W_0^{(i_0)}$ may be identified with an interval under the (left) weak order in the affine Weyl group---we explain this below.  First, note that, for any element $v \in P^{(i_0)} \cong W_0^{(i_0)}$ obtained from $\rho^{(i_0)}$ by a sequence of firings $i_1, i_2, \ldots, i_m \in I$ of amplitudes $< -1$, Proposition \ref{strloopcmpp} shows that the element $s_{i_m} s_{i_{m-1}} \cdots s_{i_1} \in W$ depends only on $v$ and not on the choice of firing sequence.  That is, we obtain an embedding $P^{(i_0)} \into W$. Let $\varphi: W_0^{(i_0)} \cong P^{(i_0)} \into W$ be the resulting composition.  This is a section of the quotient $\chi: W \onto W_0$ defined by $g \widetilde{v} = \widetilde{\chi(g) v}$ for any $v \in \R^{I_0}$ (that is, $\chi(s_i) = s_i$ for $i \in I_0$, and $\chi(s_{i_0})$ is the reflection about the maximal root of $\Gamma_0$).  Precisely, $\chi \circ \varphi: W_0^{(i_0)} \rightarrow W_0$ is the inclusion.  We will now show that the image poset $\varphi(W_0^{(i_0)}) \subset W$ is nothing but an interval in $W$ under the weak order.

We recall the definition of weak order. For $g \in W$, the \emph{length} of $g$, denoted $l(g)$, is the minimal number of simple reflections $s_i$ needed to multiply to $g$.  The \emph{left} weak order in $W$ is the ordering such that $g \leq_L h$ if and only if $l(h) = l(g) + l(hg^{-1})$, and the \emph{right} weak order in $W$ is the ordering such that $g \leq_R h$ if and only if $l(h) = l(g) + l(g^{-1}h)$.

Generally, define the \emph{numbers game ordering} to be: $v'$ is less than $v''$ if $v''$ can be obtained from $v'$ by playing the numbers game.  

\begin{lemma} Let $\Gamma$ be any graph associated to a Coxeter group $W$.  Given any configuration $v$ for which the numbers game terminates at $u = g v$ for $g \in W$, then the map $W \rightarrow \R^I, h \mapsto h u$ restricts to an isomorphism of the interval $[\id, g]_{<_L}$ with the numbers game poset from $v$ to $u$.
\end{lemma}

\begin{proof} By strong convergence, if a valid firing takes $h u$ to $s_i h u$, where $h \in W$, $u$ is dominant, then $l(s_i h) = l(h) - 1$. As a consequence, it inductively follows that it takes exactly $l(h)$ moves to take $h u$ to $u$ by playing the numbers game.  The result follows immediately.
\end{proof}

Returning to our situation, let $w^{\text{top}} \in W_0^{(i_0)}$ be the top degree element, i.e., $w^{\text{top}} \rho^{(i_0)} = -\rho^{(\iota(i_0))}$. We immediately deduce 

\begin{cor}
The isomorphism $\varphi$ takes the poset $W_0^{(i_0)}$ to the interval $[1,\varphi(w^{\text{top}})]_{<_L}$.
\end{cor}

\subsection{Triangulation of the unit hypercube in the reflection representation}

Denote by $\mathcal{K}_{\kappa} := \{v \in \mathcal{H}_\kappa \mid v_i \in [0,\kappa], \forall i \in I_0 \}$ the ``unit hypercube'' in the hyperplane $\mathcal{H}_\kappa := \{v \in \R^I \mid \delta \cdot v = \kappa\}$. Note that $\mathcal{K}_{\kappa}$ is a fundamental domain under the group generated by the translations $T_j$ used in Lemma \ref{trlem}, and its image under $\mathcal{H}_\kappa \iso \R^{I_0}$ is the hypercube $[0,\kappa]^{I_0}$.

Let us associate to the poset $W_0^{(i_0)}$ the directed graph $\Gamma(W_0^{(i_0)})$ whose vertices are elements of $W_0^{(i_0)}$ and whose directed edges are $g \rightarrow h$ such that $\widetilde{h \rho}$ is obtained from $\widetilde{g \rho}$ by firing a single vertex of amplitude $< -1$ (i.e., $l(\varphi(h)) = l(\varphi(g))+1$).

Let $C_e \subset \mathcal{H}_{\kappa}$ be the dominant Weyl chamber, i.e., $C_e = \{v \in \mathcal{H}_{\kappa} \mid v_i \geq 0, \forall i \in I\}$. To any polytope that is the union of Weyl chambers, we associate a dual directed graph, which is the usual dual graph forgetting orientation, with orientation given by $g C_e \rightarrow g s_i C_e$ when $l(g) < l(g s_i)$.

\begin{prop}
The graph $\Gamma(W_0^{(i_0)})$ is isomorphic to the dual of the triangulation of the unit hypercube $\mathcal{K}_\kappa$ by Weyl chambers.
\end{prop}

\begin{proof}
The dual of the triangulation of $\mathcal{K}_\kappa$ is the interval $[1,w']_{<_R}$ under the right weak order, where $w'$ is the longest element such that $w' C_e \subset \mathcal{K}_\kappa$.  We claim that $w' = \varphi(w^{\text{top}})^{-1}$.  Given the claim, the result follows immediately from the fact that $[1,\varphi(w^{\text{top}})]_{<_L}$ is isomorphic to $[1, w']_{<_R}$ under the inversion map (which sends the left weak order to the right weak order).

To prove the claim, first note that the Weyl chamber containing $u_\kappa$ is in $\mathcal{K}_\kappa$ and is the one incident to the corner $v$ of $\mathcal{K}_\kappa$ that is opposite to $e$, i.e., to the corner $v$ given by $v_i = \kappa$ for $i \in I_0$. (To see that $u_\kappa$ and $v$ are in the same Weyl chamber, one can take $\kappa$ very close to $1$ without changing which chamber $u_\kappa$ is in, which would make $u_\kappa$ very close to $v$.) Next, observe that, by Proposition \ref{strloopcmpp}, the element $\varphi(w^{\text{top}})(u_\kappa)$ is in the dominant Weyl chamber $C_e$. Thus, $\varphi(w^{\text{top}})^{-1}$ is the longest element which takes $C_e$ to $\mathcal{K}_\kappa$, i.e., we must have $w' = \varphi(w^{\text{top}})^{-1}$.
\end{proof}

\begin{figure}
\centerline{\scalebox{0.5}{\includegraphics{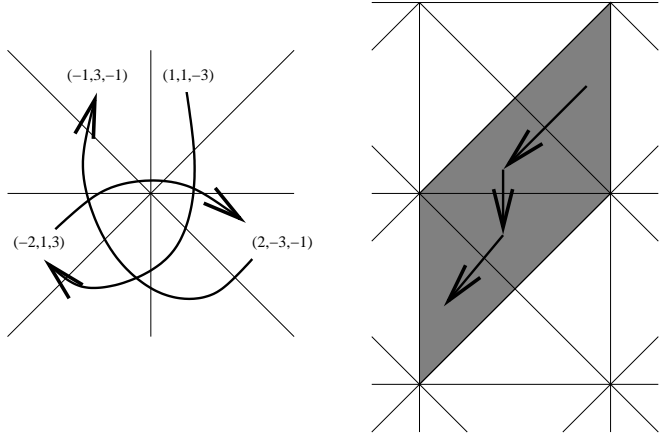}}}
\caption{The graph $\Gamma(W_0^{(i_0)})$ in $B_2$, and the unit hypercube in $\tilde{B}_2$} \label{B2Figure}
\end{figure}

In figure~\ref{B2Figure}, we demonstrate the above concepts in type $\tilde{B}_2$. On the left, we depict the graph $\Gamma(W_0^{(i_0)})$. The point $(a,b,c)$ means the point of $\R^{I}$ with those coordinates. Our convention is that the first coordinate corresponds to the root $(1,0)$ of $B_2$, the second coordinate to the root $(-1,1)$ of $B_2$, and the third coordinate to $(-1,-1)$, the negation of the longest root. On the right, we show the unit hypercube of $\tilde{B}_2$, and how $\Gamma(W_0^{(i_0)})$ occurs as the dual to this hypercube.

\section{The Hilbert polynomial of $W_0^{(i_0)}$} \label{hilbpolysec}

In what follows, we will consider any subset of $W$ as being endowed with the graded poset structure given by the \emph{right weak order}, $<_R$. Also, for any finite Coxeter group associated to a graph $\Gamma_0$ with vertex set $I_0$, let $m_1, m_2, \ldots, m_{|I_0|}$ be its Coxeter exponents.

Recall that the Hilbert polynomial $h(P;t)$ of a graded poset $P$ is defined as
\begin{equation}
h(P;t) = \sum_{d \geq 0} |\{x \in P: |x|=d\}| t^d,
\end{equation}
where $|x|$ denotes the degree of $x$. We may explicitly compute the
Hilbert polynomial of the graded poset $W_0^{(i_0)}$:

\begin{thm} \label{hilbpolythm}
The Hilbert polynomial of $W_0^{(i_0)}$ is given by
\begin{equation}
h(W_0^{(i_0)}; t) = \frac{\prod_{i \in I_0} (1 - t^{l(t_i)})}{\prod_{i=1}^{|I_0|} (1 - t^{m_i})}.
\end{equation}
\end{thm}
Here, the elements $t_i \in W$ are the translations as defined in the Lemma \ref{trlem}.

Note that, evaluating the polynomial at $t= 1$ and using that $W_0$ decomposes into isomorphic copies of $W_0^{(i_0)}$, one copy for each extending vertex, we obtain \eqref{curid}.

We remark that there is always a way to rearrange the factors in the denominator, i.e., to assign to each vertex $i \in I_0$ an exponent $m_i$, so that $\frac{1-t^{l(t_i)}}{1-t^{m_i}} = 1 + t^{m_i} + t^{2m_i} + \cdots + t^{l(t_i)-m_i}$ is a polynomial.  In some sense, this can be done uniquely: see \S \ref{explsec}.

\begin{proof}
  Let $H_+ \subset W$ be the semigroup generated by the elements $t_i$
  defined in Lemma \ref{trlem}, i.e., the elements of the form
  $T_{i_1} T_{i_2} \cdots T_{i_m} \gamma$, where $i_1, \ldots, i_m \in
  I_0$, the $T_{i_j}$'s were defined in \ref{translations}, and
  $\gamma \in \Aut(\Gamma)$. We claim that
\begin{equation} \label{wdecomp1}
W = W_0 H_+ \varphi(W_0^{(i_0)})^{-1},
\end{equation}
where $W_0 \subset W$ is the subgroup generated by the reflections
$s_j$ for $j \in I_0$, and moreover that every element $w \in W$ has a
unique decomposition as $w = w_0 h \varphi(w')^{-1}$ where $w_0 \in
W_0$, $h \in H_+$, and $w' \in W_0^{(i_0)}$, satisfying
\begin{equation} \label{wdecomp2}
l(w) = l(w_0) + l(h) + l(\varphi(w')^{-1}).
\end{equation}
As a consequence, by taking Hilbert series (using the well known formulas \cite[Theorems 7.1.5, 7.1.10]{BB} for $h(W;t)$ and $h(W_0;t)$),
\begin{equation}
h(W,t) = \frac{h(W_0;t)}{\prod_{i=1}^{|I_0|} (1-t^{m_i})} = h(W_0;t)  \frac{1}{\prod_{i\in I_0} (1-t^{l(t_i)})} h(W_0^{(i_0)}),
\end{equation}
which proves the theorem, subject to proving \eqref{wdecomp1},
\eqref{wdecomp2}, which we do now. Applying both sides to the
fundamental Weyl chamber in $\mathcal{H}_{\kappa}$, the statement is
saying that a fundamental domain for $W_0 \subset W$ in
$\mathcal{H}_{\kappa}$ is given by the image of the dominant Weyl
chamber under $\varphi(W_0^{(i_0)})^{-1} H_+$, i.e., the cone $\{v \in
\mathcal{H}_{\kappa} \mid v_i \geq 0, \forall i \in I_0\}$.  This
follows from the fact that, under the projection $\mathcal{H}_{\kappa}
\iso \R^{I_0}$ by forgetting the $i_0$-coordinate, this cone is the
preimage of the dominant $W_0$-chamber in $\R^{I_0}$.
\end{proof}

\subsection{Explicit formulas for $h(W_0^{(i_0)})$}\label{explsec}

Below we give each Dynkin graph $\Gamma_0$, with vertices $i$ labeled by two positive integers: $l(t_i)$, and a Coxeter exponent $m_i$ such that $m_i \mid l(t_i)$, so that each $m_i$ occurs once. We will need the following notation: the \emph{odd part}, $\{ m \}_2$, of $m \in \Z_+$ is the maximal odd factor of $m$.

\begin{figure}[hbt]
\begin{center}
\input{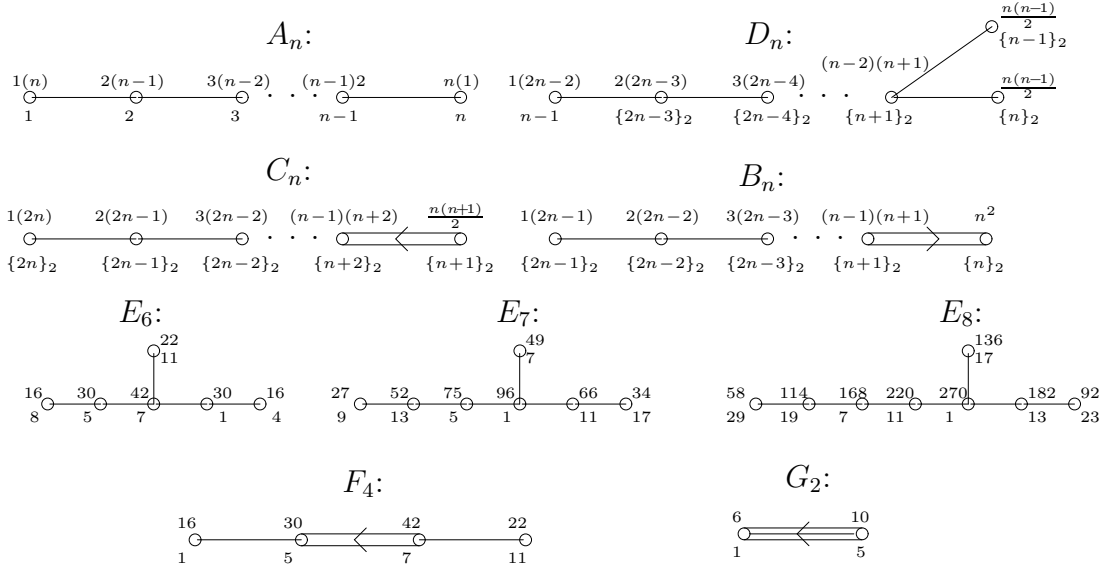}
\caption{The lengths $l(t_i)$ and Coxeter exponents $m_i$ for all Dynkin graphs}
\label{dynsfig}
\end{center}
\end{figure}

Set $Q_m(t) := \frac{t^m-1}{t-1} = 1 + t + t^2 + \cdots + t^{m-1}$. Then, by Theorem \ref{hilbpolythm},
\begin{equation} \label{hqfact}
h(W_0^{(i_0)};t) = \prod_{i \in I} Q_{\frac{l(t_i)}{m_i}}(t^{m_i}),
\end{equation}
which is a factorization of $h(W_0^{(i_0)})$ by polynomials whose nonzero coefficients are all $1$.

The top degrees of the posets $W_0^{(i_0)}$ are the degrees of the Hilbert polynomials, given from the figure by:
\begin{center}
\begin{tabular}{|c|c|c|c|c|c|c|c|c|}

$A_n$ & $B_n$ & $C_n$ & $D_n$ & $E_6$ & $E_7$ & $E_8$ & $F_4$ & $G_2$ \\ \hline
${n+2 \choose 3}$ & $\frac{n(n-1)(4n+1)}{6}$ & $\frac{n(n-1)(4n+1)}{6}$ &
 $4 {n \choose 3}$ & $120$ & $336$ & $1120$ & $86$ & $10$
\end{tabular} $\ $.
\end{center}

For the series of types $A$ and $D$, these degrees are cubic in the number of vertices; in the exceptional cases, one may find similar identities, such as $120=6 \cdot 5 \cdot 4, 336 = 8 \cdot 7 \cdot 6$, and $1120 = \frac{1}{3}(16 \cdot 15 \cdot 14)$). An upper bound on the number of valid moves in the numbers game required to go from any vector $v$ on the extended Dynkin graph satisfying $\delta \cdot v = 0$ to any element of its Weyl orbit is given by the sum of this degree and the weak order degree of $W_0$ (the latter being quadratic in the number of vertices), and is therefore cubic in the number of vertices.

Let us consider the question of how unique the assignment of the exponents $m_i$ to the vertices is such that $m_i \mid l(t_i)$.  For exceptional types, this assignment of the exponents to the vertices is unique, except in the $E_6$ case, where it is unique up to swapping $m_i$ with $m_{i'}$ when $i, i'$ are images of each other under an element of $\Aut(\Gamma_0)$. For each infinite series, one can make a uniqueness statement if one views the collection of graphs for all $n$ together.  For example, for types $A, B$, and $C$, if we label the vertices for such a series subject to the condition that the maximal segment of consecutively-numbered vertices goes to infinity, then this is the unique assignment of $m_i$ so that the $m_i$ are given by a polynomial in $n$ and $i$ or the odd part of such a polynomial. For type $D$, this is true except for the first vertex which is assigned a different polynomial, $n-1$.

We remark that, for types $B$ and $C$, even though their finite Weyl groups are identical, Hilbert polynomials $h(W_0^{(i_0)}(B_n);t) \neq h(W_0^{(i_0)}(C_n);t)$, and in particular $W_0^{(i_0)}(B_n) \not \cong W_0^{(i_0)}(C_n)$. However,  $h(W_0^{(i)}(B_n);1) = h(W_0^{(i)}(C_n);1)$; indeed, both must equal $\frac{1}{2} |W_0|$.  Moreover, $\deg h(W_0^{(i_0)}(B_n);t) = \deg h(W_0^{(i_0)}(C_n);t)$.

\subsection{Combinatorial interpretation of the Hilbert polynomial for type $A$} \label{combintsec}

In this section, we consider the case of $\tilde{A}_{n-1}$. In the proof of Theorem~\ref{hilbpolythm}, we computed the Hilbert series of the inverse image in $\mathcal{H}_{\kappa}$ of the dominant Weyl chamber for $A_{n-1}$; specifically, we showed that this series is $1/(1-t)(1-t^2)\cdots (1-t^{n-1})$. In~\cite[Section 9.4]{EE98}, Eriksson and Eriksson gave a combinatorial proof of this result. We sketch their proof, and explain how to modify it to give a combinatorial proof of Theorem~\ref{hilbpolythm}. Let here $W$ and $W_0$ be the affine and finite Weyl groups of types $\widetilde {A_{n-1}}$ and $A_{n-1}$, respectively.

Let $\tilde{S}_n$ be the set of permutations $i \mapsto s_i$ of the
integers such that $s_{i+n} = s_i +n$ and $\sum_{i=1}^n s_i =
\sum_{i=1}^n i$. Define a map $\partial : \tilde{S}_n \to \R^I$ by
$(\partial s)_i = s_i - s_{i-1}$; it is well known that this map is
injective and its image is $W \cdot (1,1,\ldots,1)$. So, we can
identify $\tilde{S}_n$ with $W$. Under this identification, the
dominant chamber consists of those permutations with $s_1 < s_2 <
\cdots < s_n$. The unit hypercube consists of those permutations
where, in addition, $s_1 + n > s_2$, $s_2+n > s_3$, \dots, and $s_{n-1}+n
> s_n$.

Given $(s_i)$ in the dominant chamber, define an $n$-tuple $(\gamma_1,
\gamma_2, \ldots, \gamma_{n-1})$ as follows: Let $i$ be an integer
between $1$ and $n-1$. Let $U_i$ be the set of integers $t$ such that
$t < s_{i+1}$ and $t \not \equiv s_{i+1}$, $s_{i+2}$, \dots, $s_n
\pmod n$. Number the elements of $U_i$ as $u_0 > u_1 > u_2 >
\ldots$. So, $s_{i}$ is in $U$; define $\gamma_i$ by $s_{i} =
u_{\gamma_i}$.

\begin{rem}
The integer $\gamma_i$ is the number of times that $n-i$ occurs in the sequence $\delta_{\bullet}$ constructed in~\cite{EE98}.
\end{rem}

It is easy to see that this is a bijection between the dominant chamber and $\Z_{\geq 0}^{n-1}$. It is also clear that the unit hypercube corresponds to the set $0 \leq \gamma_i < i$.

\begin{prop}
Let $(s_i)$ be in the dominant chamber. Under this bijection, the length of $(s_i)$ is $\sum (n-i) \gamma_i$.
\end{prop}

\begin{proof}
We claim that $s_1 +n > s_{j}$ if and only if $\gamma_{1} = \gamma_{2} = \cdots = \gamma_{j-1} =0$.  Proof: There are precisely $j-1$ elements of $U_{j-1}$ which are greater than $s_{j}-n$.  Now, $s_1$ is greater than $s_{j}-n$; if and only if $s_1$, $s_2$, \dots, $s_{j-1}$ are all greater than $s_j-n$; if and only if $s_1$, $s_2$, \dots, $s_{j-1}$ are the $j-1$ largest elements of $U_{j-1}$; if and only if $\gamma_{1} = \gamma_{2} = \cdots = \gamma_{j-1} =0$.

We prove the proposition by induction on $\sum \gamma_i$.  The result is obvious when all of the $\gamma_i$ are zero, so we may assume this is not true.  Let $\gamma_{1} = \gamma_{2} = \cdots = \gamma_{j-1} =0$ and $\gamma_{j}>0$.  So $s_{j} < s_1+n < s_{j+1}$.  Consider the element $s'$, in the dominant chamber, where $(s'_1, s'_2, \ldots, s'_n) = (s_2 -1, s_3 -1, \ldots, s_j -1, s_1 +n-1, s_{j+1} -1, \ldots, s_n-1)$. It is not hard to verify that the new $\gamma$ vector is $(0,0,\ldots,0,\gamma_j -1, \gamma_{j+1}, \ldots, \gamma_{n-1})$. Let $t$ denote $x \mapsto x+1$; conjugation by $t$ is a length preserving automorphism of $\widetilde{S}_n$.  Then, $t \circ s \circ t^{-1} = \omega \circ s' $, where $\omega$ is an $(n-j+1)$ cycle and the right hand side is length-additive, so $\ell(s) = \ell(s')+(n-j)$.
\end{proof}

So we have a bijection between the unit hypercube and $\{ 0 \} \times \{ 0,1 \} \times \cdots \times \{ 0,1,\ldots,n-2 \}$ where the element corresponding to $(\gamma_1, \gamma_2, \ldots, \gamma_{n-1})$ has length $\sum (n-i) \gamma_i$. This gives a bijective proof that the Hilbert series of the unit hypercube is $\prod_{i=1}^{n-1} \left( 1+t^{n-i} + t^{2(n-i)} + \cdots + t^{(i-1)(n-i)} \right)$.

We note that Eriksson and Eriksson also give combinatorial proofs of the Hilbert series for other classical types; these proofs might be able to be similarly adapted.

\bibliographystyle{amsalpha}
\bibliography{references}
\end{document}